\pdfoutput=1
\RequirePackage{ifpdf}
\ifpdf 
\documentclass[pdftex]{sigma}
\else
\documentclass{sigma}
\fi

\numberwithin{equation}{section}

\newtheorem{Theorem}{Theorem}[section]
\newtheorem{Lemma}[Theorem]{Lemma}
 { \theoremstyle{definition}
\newtheorem{Fact}[Theorem]{Fact}
\newtheorem{Remark}[Theorem]{Remark} }

\begin{document}

\allowdisplaybreaks

\renewcommand{\thefootnote}{$\star$}

\newcommand{\arXivNumber}{1512.07104}

\renewcommand{\PaperNumber}{049}

\FirstPageHeading

\ShortArticleName{Shell Polynomials and Dual Birth-Death Processes}

\ArticleName{Shell Polynomials and Dual Birth-Death Processes\footnote{This paper is a~contribution to the Special Issue
on Orthogonal Polynomials, Special Functions and Applications.
The full collection is available at \href{http://www.emis.de/journals/SIGMA/OPSFA2015.html}{http://www.emis.de/journals/SIGMA/OPSFA2015.html}}}

\Author{Erik A.~VAN DOORN}

\AuthorNameForHeading{E.A.~van~Doorn}

\Address{Department of Applied Mathematics, University of Twente,\\ P.O.~Box 217, 7500 AE Enschede, The Netherlands}

\Email{\href{mailto:e.a.vandoorn@utwente.nl}{e.a.vandoorn@utwente.nl}}
\URLaddress{\url{http://wwwhome.math.utwente.nl/~doornea/}}

\ArticleDates{Received January 02, 2016, in f\/inal form May 14, 2016; Published online May 18, 2016}

\Abstract{This paper aims to clarify certain aspects of the relations between birth-death
processes, measures solving a Stieltjes moment problem, and sets of parameters
def\/ining polynomial sequences that are orthogonal with respect to such a measure.
Besides giving an overview of the basic features of these relations,
revealed to a large extent by Karlin and McGregor, we investigate a duality concept
for birth-death processes introduced by Karlin and McGregor and its interpretation
in the context of shell polynomials and the corresponding orthogonal polynomials.
This interpretation leads to increased insight in duality, while it suggests a
modif\/ication of the concept of similarity for birth-death processes.}

\Keywords{orthogonal polynomials; birth-death processes; Stieltjes moment
problem; shell polynomials; dual birth-death processes; similar birth-death processes}

\Classification{42C05; 60J80; 44A60}

\renewcommand{\thefootnote}{\arabic{footnote}}
\setcounter{footnote}{0}

\section{Introduction}
\label{introduction}

In what follows a {\em measure} will always be a f\/inite positive Borel measure on
the real axis with inf\/inite support and f\/inite moments of all (positive) orders. It will be
convenient to assume throughout that the measure is normalized so that it becomes
a probability measure.
The {\em Hamburger moment problem} associated with a measure $\psi$ is said
to be {\em determined} ($\psi$ is det(H), for short) if~$\psi$ is uniquely
determined by its moments; otherwise, it is said to be {\em indeterminate}
($\psi$~is~indet(H)).
Similar terminology will be used for the {\em Stieltjes moment problem} associated
with $\psi$, in which we limit our scope to measures with support on the nonnegative
real axis, with~det(S) (indet(S)) replacing~det(H) (indet(H)).

Chihara~\cite{C68} showed that when a measure is indet(H) and has left-bounded
support, there is a unique solution
of the associated moment problem with the property that the minimum of its
support is maximal. We will refer to this solution as the {\em natural} solution.
It will be convenient to qualify a measure as {\em natural} also if it is the
solution of a determined moment problem. Note that the natural solution of an
indeterminate Hamburger moment problem may be the unique solution of a Stieltjes
moment problem.

Our point of departure is a measure $\psi$ on the {\em nonnegative} real
axis with moments
\begin{gather*}
m_n(\psi) := \int_{[0,\infty)} x^n \psi(dx), \qquad n \geq 0.
\end{gather*}
By assumption $m_0(\psi) = 1$. In what follows we allow $\psi$ to be indet(S) (and hence
indet(H)) but assume in this case that $\psi$ is the natural solution of the
associated moment problem.
The (monic) polynomials that are orthogonal with respect to $\psi$ will be denoted
by $P_n$.

As is well known there exist unique constants $c_n \in \mathbb{R}$ and
$d_{n+1} > 0$, $n\geq 1$, such that the polynomials~$P_n$ satisfy the
three-terms recurrence relation
\begin{gather}
P_n(x) = (x - c_n) P_{n-1}(x) - d_n P_{n-2}(x), \qquad n > 1,\nonumber\\
P_1(x) = x - c_1, \qquad P_0(x) = 1.\label{recP}
\end{gather}
Since the support of $\psi$ is a subset of the nonnegative real axis we actually have
$c_n>0$, and there is a more ref\/ined result (see, for instance, Chihara \cite[Corollary
to Theorem~I.9.1]{C78}) to the ef\/fect that there are numbers $\mu_0 \geq 0$ and
$\lambda_n > 0$, $\mu_{n+1}> 0$ for $n\geq 0$, such that
\begin{gather}
\label{rep}
c_n = \lambda_{n-1} + \mu_{n-1} \qquad \mbox{and} \qquad d_{n+1} = \lambda_{n-1}\mu_n, \qquad n \geq 1.
\end{gather}
(Further results in this vein can be found in \cite{C07}.)
Since $\lambda_n$ and $\mu_n$ may be interpreted as the birth rates and death rates,
respectively, of a {\em birth-death process} on the nonnegative integers, we will
refer to a collection of such constants as a {\em set of birth and death rates} (or a~{\em rate set}, for short).
More information on birth-death processes and their rates will be given in later sections,
but at this stage we note that, by Karlin and McGregor \cite[Lemma~1 and Lemma~6 on p.~527]{K57a} (see also \cite[Theorem~1.3]{C07}), we must have
$\mu_0 = 0$ unless $\psi$ has a f\/inite moment of order~$-1$, that is,
\begin{gather}
\label{m-1}
m_{-1}(\psi) := \int_{[0,\infty)} x^{-1} \psi(dx) < \infty,
\end{gather}
in which case $\mu_0$ may be any number in the interval $[0, 1/m_{-1}(\psi)]$.
Evidently, once $\mu_0$ has been chosen, the other rates are f\/ixed.

In the remainder of this section we will assume that \eqref{m-1} is satisf\/ied, so that,
in particular, $\psi(\{0\}) = 0$. Def\/ining the measure $\phi^{(0)}$ by
\begin{gather}
\label{phi0}
\phi^{(0)}([0,x]) := \frac{1}{m_{-1}(\psi)} \int_{[0,x]}y^{-1}\psi(dy), \qquad x \geq 0,
\end{gather}
and letting, for $a > 0$,
\begin{gather}
\label{phia}
\phi^{(a)} := \frac{1}{a+1}\big(a\delta_0 + \phi^{(0)}\big),
\end{gather}
where $\delta_0$ is the Dirac measure with mass $1$ at $0$, we observe that, for any
$a \geq 0$, $\phi^{(a)}$ is a~probability measure on the nonnegative real axis.
As a consequence there exists a sequence of (monic) polynomials $\big\{S^{(a)}_n\big\}$
that are orthogonal with respect to~$\phi^{(a)}$. Since, for all $a \geq 0$,
\begin{gather*}
\psi([0,x]) = (a+1)m_{-1}(\psi) \int_{[0,x]} y\phi^{(a)}(dy), \qquad x \geq 0,
\end{gather*}
we will, following Chihara \cite{C01}, refer to the polynomials $S^{(a)}_n$ as
{\em shell polynomials} corresponding to the orthogonal polynomial sequence~$\{P_n\}$. In the terminology of \cite[Section~I.7]{C78} the polyno\-mials~$P_n$
are, for any $a \geq 0$, the {\em kernel polynomials} with $K$-parameter $0$
corresponding to $\big\{S^{(a)}_n\big\}$.

Information on the status of the moment problem associated with $\phi^{(a)}$ is
given in the next theorem.
\begin{Theorem}[Chihara \protect{\cite[Theorem~2]{C62}}]\label{det}\quad
\begin{enumerate}\itemsep=0pt
\item[$(i)$] The measure $\phi^{(0)}$ is {\rm det(H)}.
\item[$(ii)$] The measure $\phi^{(a)}$, $a > 0$, is {\rm det(S)} if and only if the measure
$\psi$ is {\rm det(S)}.
\end{enumerate}
\end{Theorem}

Evidently, $\phi^{(a)}$ cannot be a natural measure if it is indet(S), but we
will see that $\phi^{(a)}$ nevertheless has a certain signif\/icance in this case,
in particular in the context of birth-death processes.
Before introducing our f\/indings it will be useful to state already some facts about
the relations between birth-death processes, rate sets, and measures that will
be expounded in Section~\ref{birthdeath}.

\begin{Fact}\label{fact1}
A birth-death process uniquely def\/ines a rate set.
\end{Fact}

\begin{Fact}\label{fact2}
A birth-death process uniquely def\/ines a measure on $[0,\infty)$
through Karlin and McGregor's representation formula~\eqref{spectralrep} for the
transition functions of a birth-death process. This measure is natural or of a type
known as {\em Nevanlinna extremal}.
\end{Fact}

\begin{Fact}\label{fact3}
A rate set $\{\lambda_n,\mu_n\}$ uniquely def\/ines, through \eqref{rep},
a {\em natural} measure on $[0,\infty)$ with respect to which the polynomials
$P_n$ of \eqref{recP} are orthogonal. Conversely, a natural measure on~$[0,\infty)$,
with monic orthogonal polynomials $P_n$ and moment $m_{-1}$ of order $-1$,
def\/ines, through~\eqref{recP} and~\eqref{rep}, a~rate set $\{\lambda_n,\mu_n\}$, which,
if $m_{-1}=\infty$, is unique and satisf\/ies $\mu_0=0$. If $m_{-1}<\infty$ the
measure def\/ines an inf\/inite family of rate sets indexed by the value of $\mu_0$,
which can be any number in the interval $[0, 1/m_{-1}]$.
\end{Fact}

\begin{Fact}\label{fact4}
 A rate set $\{\lambda_n,\mu_n\}$ uniquely def\/ines a birth-death
process if and only if at least one of the following conditions prevails:
\begin{enumerate}\itemsep=0pt
\item[(i)] the natural measure def\/ined by the rate set is det(S);
\item[(ii)] $\mu_0>0$ and the natural measure def\/ined by the rate set has
$m_{-1} = \mu_0^{-1}$.
\end{enumerate}
Otherwise, there is an inf\/inite, one-parameter family of birth-death processes with
the given rates. Two members of this family may be identif\/ied as extreme, and are
known as the {\em minimal process} (associated with the natural measure) and the
{\em maximal process}.
\end{Fact}

Fact~\ref{fact1} is a trivial consequence of the def\/inition of a birth-death process in
Section~\ref{basic}; Facts~\ref{fact2} and~\ref{fact3} follow from the seminal work of Karlin and
McGregor \cite{K57a} (Fact~\ref{fact3} also summarizes earlier observations in this section);
for Fact~\ref{fact4} we refer to~\cite{K57a} again and, for the second part, to~\cite{D87}.
Note that, by Fact~\ref{fact3}, our measure $\psi$ def\/ines inf\/initely many rate sets, since
$m_{-1}(\psi) < \infty$ by assumption.

Our main goal in this paper is to interpret and characterize the measures $\phi^{(a)}$,
$a \geq 0$, and the relation between the measures $\psi$ and~$\phi^{(a)}$, in the
context of birth-death processes. Concretely, we will display a one-to-one
correspondence between the measures~$\phi^{(a)}$, $a\geq 0$, and the rate sets with
$\mu_0>0$ def\/ined by $\psi$ (to which we will refer as {\em $\psi$-rate sets} for
short). Moreover, if $\phi^{(a)}$ is det(S), then the birth-death process (uniquely)
def\/ined by $\phi^{(a)}$ and the (unique) birth-death process whose rate set is the
$\psi$-rate set corresponding to $\phi^{(a)}$ are shown to be {\em dual} to each
other in sense of Karlin and McGregor~\cite[Section~6]{K57b}. As a result the duality
concept for birth-death processes can be extended to families of birth-death processes
that are {\em similar} in the sense of~\cite{L00}, after a slight modif\/ication of
the def\/inition of similarity.

When $\phi^{(a)}$ is indet(S) -- and hence, by Theorem \ref{det}, $a>0$ -- the
situation is more complicated since the duality concept for birth-death processes can
be applied only to minimal and maximal processes when a rate set does not def\/ine a
birth-death process uniquely (see \cite{D87}).
However, we will see that in this case $\phi^{(a)}$ is Nevanlinna extremal (as stated
already by Berg and Christiansen~\cite{B11}) and corresponds to a {\em maximal}
birth-death process, which happens to be dual to the {\em minimal} process whose
rate set is the $\psi$-rate set corresponding to~$\phi^{(a)}$. We will also present a~counterpart of this result.

In the next section we will collect some further notation, terminology and
preliminary results about shell polynomials, rate sets, and measures, while in Section
\ref{birthdeath} the relevant properties of birth-death processes are set forth and put
in proper perspective. Our f\/indings are detailed in Section~\ref{results}.

\section{Preliminaries}

\subsection{Shell polynomials and rate sets}
\label{shellpolynomials}

Applying \cite[Corollary to Theorem I.9.1]{C78} to the polynomials $S^{(a)}_n$, we
conclude that, for any $a \geq 0$, there exist constants $\mu^{(a)}_0 \geq 0$ and
$\lambda^{(a)}_n > 0$, $\mu^{(a)}_{n+1}> 0$ for $n\geq 0$, such that
\begin{gather}
S^{(a)}_n(x) = \big(x - \lambda^{(a)}_{n-1}-\mu^{(a)}_{n-1}\big) S^{(a)}_{n-1}(x) -
\lambda^{(a)}_{n-2}\mu^{(a)}_{n-1} S^{(a)}_{n-2}(x), \qquad n > 1,\nonumber\\
S^{(a)}_1(x) = x - \lambda^{(a)}_0 - \mu^{(a)}_0, \qquad S^{(a)}_0(x) = 1.\label{recS}
\end{gather}
If $\phi^{(a)}$ is natural and $m_{-1}(\phi^{(a)}) = \infty$ (so in particular
if $\phi^{(a)}$ is det(S) and $a > 0$) then, by \cite[Lemma~1 and Lemma~6 on p.~527)]{K57a} again, we must have $\mu_0^{(a)} = 0$. In other circumstan\-ces~$\mu^{(a)}_0$ may also be chosen positive, but it will be convenient to set
$\mu_0^{(a)} = 0$ by def\/inition in what follows.

We can now relate the parameters in the recurrence relation~\eqref{recS}
for the polyno\-mials~$S_n^{(a)}$ to the parameters $c_n$ and $d_n$ in the recurrence
relation~\eqref{recP}. Indeed, since the polyno\-mials~$P_n$ are the kernel polynomials
with $K$-parameter $0$ corresponding to $\big\{S_n^{(a)}\big\}$, we have, by
\cite[Theo\-rem~I.9.1]{C78},
\begin{gather*}
c_n = \lambda_{n-1}^{(a)} + \mu_n^{(a)} \qquad \mbox{and} \qquad
	d_{n+1} = \lambda_n^{(a)}\mu_n^{(a)}, \qquad n \geq 1.
\end{gather*}
Subsequently def\/ining birth rates $\lambda_n$ and death rates $\mu_n$ by
\begin{gather}
\label{dual}
\lambda_n = \mu_{n+1}^{(a)} \qquad \mbox{and} \qquad \mu_n = \lambda_n^{(a)}, \qquad n \geq 0,
\end{gather}
it follows that we have regained \eqref{rep}. So we see that we can parametrize the
birth and death rates in the representation~\eqref{rep} by the value of~$\mu_0$, but also
by the size $a$ of the atom at $0$ of the measure $\phi^{(a)}$ of~\eqref{phi0} and~\eqref{phia}, since the value of $a$ uniquely identif\/ies the shell polynomials~$S_n^{(a)}$
corresponding to~$\{P_n\}$, and hence, through~\eqref{recS} (where $\mu_0^{(a)} = 0$) and~\eqref{dual}, the birth and death rates.

The next theorem gives an explicit one-to-one relation between $\mu_0$ and~$a$,
and shows that the question of whether the alternative representation yields {\em all}
possibilities can be answered in the af\/f\/irmative, provided we allow $0 \leq a \leq \infty$
and interpret~$\lambda_n^{(\infty)}$ and $\mu_n^{(\infty)}$ as limits as $a\to\infty$ of the
corresponding quantities with superindex~$^{(a)}$. We will see in Section~\ref{chainsequences} that the theorem is an immediate corollary of a~theorem of Chihara~\cite{C62}.

\begin{Theorem}
\label{mu0a}
Let $\psi$ be a natural measure satisfying \eqref{m-1} and let $\mu_0$ be determined
by~$a$ via~\eqref{phi0}, \eqref{phia}, \eqref{recS} $($with $\mu_0^{(a)}=0)$ and~\eqref{dual}. Then, for $0 \leq a \leq \infty$,
\begin{gather*}
\mu_0 = \frac{1}{(a + 1)m_{-1}(\psi)},
\end{gather*}
whence $\mu_0$ can have any value in the interval $[0,1/m_{-1}(\psi)]$.
\end{Theorem}

Note that, as $a \to \infty$, $\phi^{(a)}$ converges strongly to $\delta_0$,
so we cannot (and need not) extend the
def\/inition of~$S_n^{(a)}$ to include the case $a=\infty$. An interpretation of
$\lambda_n^{(\infty)}$ and $\mu_n^{(\infty)}$ as birth and death rates of a
birth-death process on the nonnegative integers is possible, but does not f\/it
in the setting described around~\eqref{rep} since
$\lambda_0^{(\infty)} = \mu_0^{(\infty)} = 0$.

Let us mention at this point that \eqref{dual} displays the duality concept
for birth-death processes that will be further discussed in Section~\ref{dualbirthdeath} and plays a crucial role in Section~\ref{results}.

\subsection{Chain sequences and rate sets}
\label{chainsequences}

We f\/irst recall some def\/initions and basic results (see Chihara \cite[Section~III.5]{C78}
and~\cite{C90} for more information). A sequence $\{a_n\}_{n=1}^{\infty}$ is a
{\it chain sequence} if there exists a second sequence $\{g_n\}_{n=0}^{\infty}$
such that
\begin{alignat*}{3}
& (i) \quad && 0 \leq g_0 < 1, \qquad 0 < g_n < 1, \qquad n \geq 1, &\\
& (ii) \quad && a_n = (1 - g_{n-1})g_n, \qquad n \geq 1.&
\end{alignat*}
The sequence $\{g_n\}$ is called a {\it parameter sequence} for $\{a_n\}$.
If both $\{g_n\}$ and $\{h_n\}$ are parameter sequences for $\{a_n\}$, then
\begin{gather*}
g_n < h_n, \qquad n \geq 0 \quad \Longleftrightarrow \quad g_0 < h_0.
\end{gather*}
Every chain sequence $\{a_n\}$ has a {\it minimal} parameter sequence,
uniquely determined by the condition $g_0 = 0$, and a~{\it maximal}
parameter sequence~$\{M_n\}$, characterized by the fact that $M_0 > g_0$
for any other parameter sequence~$\{g_n\}$. For every $x$, $0 \leq x \leq M_0,$
there is a unique parameter sequence~$\{g_n\}$ for $\{a_n\}$ such that $g_0 = x.$

Linking the parameters in the three-terms recurrence relation~\eqref{recP} to
birth and death rates is an alternative for the approach involving chain sequences
chosen by Chihara in, for instance, \cite{C62,C78}. Indeed, letting
\begin{gather*}
a_n = \frac{d_{n+1}}{c_n c_{n+1}}, \qquad n \geq 1,
\end{gather*}
we see that the sequence $\{a_n\}_{n=1}^\infty$ is a chain sequence, since
$a_n = (1-g_{n-1})g_n$ if we choose
\begin{gather} \label{repg}
g_n = \frac{\mu_n}{\lambda_n+\mu_n}, \qquad n \geq 0,
\end{gather}
for any set of birth rates $\lambda_n$ and death rates $\mu_n$ satisfying \eqref{rep}.
So~\eqref{repg} gives a one-to-one correspondence between a parameter sequence for
the chain sequence $\{a_n\}$ and a rate set satisfying~\eqref{rep}.
Since $0 \leq \mu_0 \leq 1/m_{-1}(\psi)$, we can also characterize the maximal parameter
sequence for $\{a_n\}$ by
\begin{gather*}
M_0 = \frac{1}{c_1m_{-1}(\psi)}.
\end{gather*}

Invoking \cite[Theorem~2]{C62} we can now conclude that it implies Theorem~\ref{mu0a},
for on comparing our~\eqref{phia} with \cite[equation~(3.3)]{C62} and making
appropriate identif\/ications, we f\/ind that $a = (\mu_0 m_{-1}(\psi))^{-1} - 1$, as required.

\subsection{Spectral properties and Nevanlinna extremal measures}
\label{spectral}

In this subsection we will introduce some notation and terminology concerning
the {\em natural} measure $\psi$ introduced in Section~\ref{introduction} and,
if~$\psi$ is indet(S), related measures called {\em Nevanlinna extremal}.

Of interest to us will be the quantities $\xi_i,$ recurrently def\/ined by
\begin{gather}
\label{xi1}
\xi_1 := \inf \operatorname{supp}(\psi),
\end{gather}
and
\begin{gather}
\label{xii}
\xi_{i+1} := \inf \{\operatorname{supp}(\psi)\cap (\xi_i,\infty)\}, \qquad i \geq 1.
\end{gather}
where supp($\psi$) denotes the support (or {\em spectrum}) of the measure~$\psi$.
We further def\/ine
\begin{gather}
\label{sigma}
\sigma := \lim_{i\to\infty} \xi_i,
\end{gather}
the f\/irst accumulation point of $\operatorname{supp}(\psi)$ if it exists, and inf\/inity otherwise.
So $(\operatorname{supp}\psi)$ is discrete with no f\/inite limit point if and only if $\sigma = \infty$.
It is clear from the def\/inition of~$\xi_i$ that, for all $i \geq 1$,
\begin{gather*}
\xi_{i+1} \geq \xi_i \geq 0,
\end{gather*}
and
\begin{gather*}
\xi_i = \xi_{i+1} \quad \Longleftrightarrow \quad \xi_i = \sigma.
\end{gather*}
Note that we must have $\sigma = 0$ if $\xi_1 = 0$ and $\psi(\{0\}) = 0$. Also, $\psi$
must be det(S) if $\xi_1 = 0$.

From Karlin and McGregor \cite{K57a} (see also Chihara~\cite{C82}) we know that
\begin{gather}
\label{detS}
\psi \quad \mbox{is~indet(S)} \quad \Longleftrightarrow \quad
\sum_{n=0}^\infty \left(\pi_n + \frac{1}{\lambda_n\pi_n}\right) < \infty,
\end{gather}
where
\begin{gather}
\label{pi}
\pi_0 := 1 \qquad \mbox{and} \qquad \pi_n := \frac{\lambda_0 \lambda_1 \cdots
\lambda_{n-1}} {\mu_1 \mu_2 \cdots \mu_n}, \qquad n\geq 1,
\end{gather}
and $\{\lambda_n,\mu_n\}$ is the rate set with $\mu_0 = 0$ satisfying~\eqref{rep}.
We note, parenthetically, that for a rate set with $\mu_0 > 0$ the right-hand side of
\eqref{detS} is suf\/f\/icient, but not necessary for the corresponding natural measure to be
indet(S) (see~\cite{C82}).

It is well known that $\sigma = \infty$ if $\psi$ is indet(S).
(A~{\em necessary} and suf\/f\/icient condition for $\sigma = \infty$ in terms of any rate set
satisfying~\eqref{rep} has recently been revealed in~\cite{D15a}.)
Moreover, if $\psi$ is indet(S) there are inf\/initely many solutions of the Stieltjes moment
problem associated with $\psi$. We shall be interested in particular in solutions known as
{\em Nevanlinna extremal} (or {\em $N$-extremal}, for short), which may be def\/ined
as follows (see, for example, Berg and Valent \cite[Section~1]{B95} and, for more
background information, Shohat and Tamarkin \cite[pp.~51--60]{S43} and Berg and
Valent \cite[Section~2]{B94}). Let
\begin{gather*}
\rho(x) := \left\{\sum_{n=0}^\infty p_n^2(x)\right\}^{-1}, \qquad x \in \mathbb{R},
\end{gather*}
where $p_n(x)$ are the {\em orthonormal} polynomials corresponding to~$\psi$.
Then~$\rho(x)$ is positive for all real~$x$ and equals, if $x \geq 0$, the maximal mass any
solution can concentrate at~$x$. Supposing that a solution
of the Stieltjes moment problem locates positive mass at the point~$x$, then that solution
is an $N$-extremal solution if and only if the point mass at~$x$ equals~$\rho(x)$.

Some pertinent properties of $N$-extremal solutions are the following. There is a one-to-one
correspondence between the real numbers in the interval $[0,\xi_1]$ and the $N$-extremal
solutions of the Stieltjes moment problem associated with $\psi$. For $\xi \in [0, \xi_1]$ we
denote the corresponding $N$-extremal solution by~$\psi_\xi$. The spectrum of $\psi_\xi$ is
discrete and consists of the point $\xi$ and exactly one point in each of the intervals
$(\xi_i,\xi_{i+1}]$, $i \geq 1$. Evidently, we have $\psi_{\xi_1} = \psi$ and
supp($\psi_{\xi_1}$) $= \{\xi_1,\xi_2,\dots\}$. The spectral points of two dif\/ferent
$N$-extremal solutions strictly separate each other.

We f\/inally note that the natural measure $\psi$ can be identif\/ied with the solution of the
associated Stieltjes moment problem that is related to the {\em Friedrichs extension} of
some semi-bounded operator (see Pedersen~\cite{P95} for details). The parametrization in~\cite{P95} (and in~\cite{B94,B95}) of the $N$-extremal solutions of
the indeterminate Stieltjes moment problem associated with $\psi$ dif\/fers from ours and is
ef\/fectuated by a~number in an interval $[\alpha, 0]$; the solution corresponding
to the parameter value~$\alpha$ ($<0$) is our natural solution~$\psi$ ($ = \psi_{\xi_1}$).

\section{Birth-death processes}
\label{birthdeath}

\subsection{Basic properties}
\label{basic}

In this paper a {\em birth-death process} $\mathcal{X} \equiv \{X(t),\, t \geq 0\}$,
say, will always be a continuous-time Markov chain taking values in
$\mathcal{N} := \{0,1,\ldots\}$ with the property that only transitions to
neighbouring states are permitted. The process has upward transition
(or {\em birth}) rates $\lambda_n$, $n \in \mathcal{N}$, and downward transition
(or {\em death}) rates, $\mu_n$, $n \in \mathcal{N}$, all strictly positive
except~$\mu_0$, which might be equal to~0. When $\mu_0 = 0$ the process is
irreducible, but when $\mu_0 > 0$ the process may escape from $\mathcal{N}$,
via $0$, to an absorbing state~$-1$. The $q$-matrix of transition rates
of $\mathcal{X}$, restricted to the states in~$\mathcal{N}$, will be denoted by~$Q$, that is,
\begin{gather} \label{Q}
Q = \begin{pmatrix}
-(\lambda_0 + \mu_0) & \lambda_0 & 0 & 0 & 0 & \ldots\cr
 \mu_1 & -(\lambda_1 + \mu_1) & \lambda_1 & 0 & 0 & \ldots\cr
 0 & \mu_2 & -(\lambda_2 + \mu_2) & \lambda_2 & 0 & \ldots\cr
 \ldots & \ldots & \ldots & \ldots & \ldots & \ldots\cr
 \ldots & \ldots & \ldots & \ldots & \ldots & \ldots\cr
\end{pmatrix},
\end{gather}
and $\mathcal{X}$ will be referred to as a $Q$-process. The process $\mathcal{X}$
will be identif\/ied with its transition functions
\begin{gather*}
p_{ij}(t) := \Pr\{X(t)=j\,|\,X(0)=i\}, \qquad i,j \in \mathcal{N}, \qquad t \geq 0,
\end{gather*}
and we write $P(\cdot) := (p_{ij}(\cdot)$, $i,j \in \mathcal{N})$. Besides the usual
probabilistic requirements and the {\em Chapman--Kolmogorov equations}
\begin{gather*}
P(s+t) = P(s)P(t), \qquad s \geq 0, \qquad t \geq 0,
\end{gather*}
imposed by the Markov property, the transition functions of $\mathcal{X}$ will
be assumed to satisfy both the {\em Kolmogorov backward equations}
\begin{gather*}
P^\prime (t) = QP(t), \qquad t \geq 0,
\end{gather*}
and {\em forward equations}
\begin{gather*}
P^\prime (t) = P(t)Q, \qquad t \geq 0,
\end{gather*}
with initial condition $P(0) = I$, the identity matrix.
It follows in particular that $P^\prime(0) = Q$, establishing Fact~\ref{fact1}.
We refer to Anderson~\cite{A91} for more information on continuous-time Markov
chains in general and birth-death processes in particular.

A matrix of the type~\eqref{Q} is always the $q$-matrix of a birth-death process,
but not necessarily of a unique process. Karlin and McGregor~\cite{K57a} have
shown that the $Q$-process $\mathcal{X}$ is uniquely determined by~$Q$~-- that is,
by its rates~-- if and only if the series
\begin{gather}
\label{unique}
\sum_{n=0}^\infty \left(\pi_n + \frac{1}{\lambda_n\pi_n}\right),
\end{gather}
where $\pi_n$ is given by \eqref{pi}, diverges. If $\mu_0=0$ then, in view
of \eqref{detS} (where $\mu_0=0$ is assumed), the series diverges if and only
if~$\psi$ is det(S), where $\psi$ denotes the (natural) measure def\/ined by
the rate set~$\{\lambda_n,\mu_n\}$.
If $\mu_0 > 0$ then, by \cite[Theorem~15]{K57a}, the series \eqref{unique}
diverges if and only if $\psi$ is det(S) {\em or} $m_{-1}(\psi) = 1/\mu_0$.

If the series \eqref{unique} converges there is an inf\/inite, one-parameter
family of $Q$-processes, which includes two members -- the {\em minimal}
and the {\em maximal} $Q$-process -- with matrices of transition functions
$P^{\min}(\cdot)$ and $P^{\max}(\cdot)$ that are uniquely def\/ined by the requirement
that any $Q$-process with matrix of transition functions $P(\cdot)$ satisf\/ies
\begin{gather*}
P^{\min}(t) \leq P(t) \leq P^{\max}(t), \qquad i,j \in \mathcal{N}, \qquad t \geq 0,
\end{gather*}
where $\leq$ denotes componentwise inequality. After introducing duality for
birth-death processes in the next subsection we will able to identify the parameter
characterizing the individual $Q$-processes.

Given the birth rates $\lambda_n$ and death rates $\mu_n$ of $\mathcal{X}$ we can
def\/ine positive numbers~$c_n$ and~$d_n$ by~\eqref{rep} and, subsequently,
polynomials~$P_n$ by the recurrence relation \eqref{recP}. By Favard's theo\-rem
the polynomials~$P_n$ are orthogonal with respect to a f\/inite positive Borel measure
on the real axis (with f\/inite moments of all positive orders), and it is shown in~\cite{K57a} and~\cite{C62} that, in fact, there is such a measure with
support on the {\em nonnegative} real axis. As before we will assume that
the measure is normalized to be a probability measure.
So we conclude that a set of birth and death rates uniquely def\/ines a
{\em natural} measure on the nonnegative real axis, thus conf\/irming the f\/irst
part of Fact~\ref{fact3}.

Actually, the natural measure that is def\/ined by the rates $\lambda_n$ and $\mu_n$~-- and hence by the matrix~$Q$~-- is precisely the measure $\psi$ appearing in
Karlin and McGregor's~\cite{K57a} spectral representation for the transition
functions of the unique (if the series \eqref{unique} diverges) or minimal (if
the series~\eqref{unique} converges) $Q$-process, namely,
\begin{gather}
\label{spectralrep}
p_{ij}(t) = (-1)^{i+j} \prod_{k=1}^j\frac{1}{\lambda_{k-1}\mu_k}
\int_0^\infty e^{-xt}P_i(x)P_j(x)\psi(dx),\qquad i,j \in \mathcal{N}, \qquad t \geq 0,
\end{gather}
where an empty product is def\/ined to be~$1$. If the series~\eqref{unique}
converges the representation~\eqref{spectralrep} still holds for {\em any}
$Q$-process, provided $\psi$ is replaced by the appropriate $N$-extremal
solution of the Stieltjes moment problem associated with the rate set.
If $\mu_0=0$ every $N$-extremal measure~$\psi_\xi$, $0 \leq \xi \leq \xi_1$,
corresponds to a birth-death process. The $N$-extremal measures corresponding
to a birth-death process with $\mu_0 > 0$ will be identif\/ied in the next
subsection. In any case, the preceding remarks conf\/irm Fact~\ref{fact2}.

For completeness' sake we recall from Section~\ref{introduction} that a natural
measure on the nonnegative real axis corresponds to an inf\/inite family of rate
sets, indexed by the value of~$\mu_0$, if (and only if) the measure has a f\/inite
moment of order~$-1$.

\subsection{Dual birth-death processes}
\label{dualbirthdeath}

Our point of departure in this subsection is a birth-death process $\mathcal{X}$
that is uniquely def\/ined by its birth rates $\lambda_n$ and death rates $\mu_n$,
where $\mu_0 > 0$. Following Karlin and McGregor \cite[Section 6]{K57b},
we def\/ine the process $\mathcal{X}^d$ to be a birth-death process on $\mathcal{N}$
with birth rates $\lambda_n^d$ and death rates $\mu_n^d$ given by $\mu_0^d := 0$ and
\begin{gather}
\label{dualrates}
\lambda_n^d := \mu_n,\qquad \mu_{n+1}^d := \lambda_n, \qquad n \geq 0.
\end{gather}
Accordingly, we let
\begin{gather*}
\pi_0^d := 1 \qquad \mbox{and}\qquad \pi_n^d :=
\frac{\lambda^d_0\lambda^d_1\cdots\lambda^d_{n-1}}{\mu^d_1\mu^d_2\cdots\mu^d_n}
= \frac{\mu_0\mu_1\cdots\mu_{n-1}}{\lambda_0\lambda_1\cdots\lambda_{n-1}}, \qquad n\geq 1,
\end{gather*}
and note that
\begin{gather*}
\pi_{n+1}^d = \mu_0(\lambda_n\pi_n)^{-1} \qquad \mbox{and}\qquad
\big(\lambda^d_n\pi^d_n\big)^{-1} = \mu_0^{-1}\pi_n, \qquad n\geq 0.
\end{gather*}
Hence divergence of the series~\eqref{unique} is equivalent to divergence
of the series
\begin{gather}
\label{dualunique}
\sum_{n=0}^\infty \left( \pi^d_n + \frac{1}{\lambda^d_n\pi^d_n}\right),
\end{gather}
so that $\mathcal{X}^d$ is uniquely def\/ined by its rates if and
only if $\mathcal{X}$ is uniquely def\/ined by its rates.
So within the setting of birth-death processes that are uniquely def\/ined by their
rates, the mapping~\eqref{dualrates} establishes a one-to-one correspondence
between processes with $\mu_0 = 0$ and those with $\mu_0 > 0$. The processes~$\mathcal{X}$ and $\mathcal{X}^d$ are therefore called each other's {\em dual}.

The transition functions of $\mathcal{X}^d$ satisfy a representation
formula analogous to~\eqref{spectralrep}, involving birth-death polynomials~$P_n^d$ and a unique natural probability measure~$\psi^d$ on the nonnegative real axis with respect to which the polynomials~$P_n^d$ are orthogonal.
Still assuming divergence of~\eqref{unique} (and hence of~\eqref{dualunique}),
we have, by \cite[Lemma~3]{K57a},
\begin{gather}
\label{dualpsi}
\psi^d([0,x]) = 1 - \mu_0 m_{-1}(\psi) + \mu_0\int_{[0,x]} y^{-1}\psi(dy), \qquad x \geq 0,
\end{gather}
where $\psi$ is the (natural) measure def\/ined by $\mathcal{X}$ (which must have
$m_{-1}(\psi)<\infty$ since $\mu_0>0$).
With $\xi^d_i$ and~$\sigma^d$ denoting the quantities
def\/ined by~\eqref{xi1},~\eqref{xii} and~\eqref{sigma} if we replace~$\psi$ by~$\psi^d$, we thus have $\sigma^d = \sigma$,
\begin{gather*}
\xi^d_1 = 0 \qquad \mbox{and} \qquad \xi^d_{i+1} = \xi_i,~i>1, \qquad \mbox{if}\qquad \mu_0 m_{-1}(\psi)<1,
\end{gather*}
and
\begin{gather*}
\xi^d_i = \xi_i, \qquad i\geq 1, \qquad \mbox{if} \qquad \mu_0 m_{-1}(\psi)=1.
\end{gather*}
Interestingly, with $p_{ij}^d(t)$ denoting the transition functions of the dual
process, we also have
\begin{gather}\label{dualp}
\sum_{j\geq k} p_{ij}^d(t) = \sum_{j<i} p_{k-1,j}(t), \qquad i, k \in \mathcal{N},\qquad t \geq 0.
\end{gather}
provided the summations are interpreted to include probability mass, if any, having
escaped from~$\mathcal{N}$ to the absorbing state~$-1$ or to inf\/inity (see~\cite{D87}
and the references there for details). This property makes duality a useful
tool in the analysis of birth-death processes (see, for example,~\cite{D15b}).

If the series~\eqref{unique} (and hence the series \eqref{dualunique}) converges,
the situation is more complicated since the rate sets $\{\lambda_n,\mu_n\}$ and
$\{\lambda_n^d,\mu_n^d\}$ are associated with inf\/inite families of birth-death
processes. The following facts have been established in~\cite{D87}. First,
there is the separation result
\begin{gather*}
0 < \xi_i^d < \xi_i < \xi_{i+1}^d, \qquad i \geq 1,
\end{gather*}
where $\xi_i$ and $\xi^d_i$ now represent the spectral points of the {\em natural}
measures $\psi = \psi_{\xi_1}$ and $\psi^d = \psi^d_{\xi_1^d}$ that are uniquely
def\/ined by the rate sets $\{\lambda_n,\mu_n\}$ and $\{\lambda_n^d,\mu_n^d\}$, respectively.
(Recall that $\sigma = \sigma^d = \infty$.)

Secondly, the $N$-extremal measure $\psi_\xi$ is associated with a birth-death process
(in the sense of Section~\ref{basic}) if and only if
$\xi_1^d \leq \xi \leq \xi_1$, while the $N$-extremal solution $\psi_\xi^d$ is associated
with a birth-death process for all $\xi$ satisfying $0 \leq \xi \leq \xi_1^d$.
The birth-death processes associated with the $N$-extremal solutions $\psi_{\xi_1} = \psi$
and $\psi_{\xi_1^d}^d = \psi^d$ are {\em minimal} processes, whereas the birth-death
processes corresponding to the $N$-extremal solutions~$\psi_{\xi_1^d}$ and~$\psi_0^d$ are {\em maximal} processes.

Thirdly, \eqref{dualpsi} (and~\eqref{dualp}) remain valid if (and only if) either~$\psi^d$ is replaced by~$\psi^d_0$ or~$\psi$ by~$\psi_{\xi_1^d}$, that is, we have
\begin{gather}
\label{dualpsi1}
\psi^d_0([0,x]) = 1 - \mu_0 m_{-1}(\psi) + \mu_0\int_{[0,x]} y^{-1}\psi(dy), \qquad x \geq 0,
\end{gather}
and
\begin{gather*}
\psi^d([0,x]) = 1 - \mu_0 m_{-1}(\psi_{\xi_1^d}) + 	\mu_0\int_{[0,x]} y^{-1}\psi_{\xi_1^d}(dy), \qquad x \geq 0.
\end{gather*}
It follows that the duality concept for rate sets can be extended to birth-death
processes also if they are not uniquely def\/ined by their rates, provided one
restricts oneself to minimal and maximal processes, and links a minimal process to a maximal process.

We f\/inally remark that a probabilistic interpretation of minimal and maximal
processes involves the character of the boundary at inf\/inity (which is not specif\/ied
by the rates). This boundary may be completely absorbing (the minimal process),
completely ref\/lecting (the maximal process), or something in between. Evidently,
the distinction is relevant only if the process can {\em explode}, that is, reach inf\/inity in f\/inite time.

\subsection{Similar birth-death processes}
\label{similarbirthdeath}

Consider, besides the birth-death process $\mathcal{X}$ of Section~\ref{basic},
another birth-death process~$\tilde{\mathcal{X}}$, with birth rates $\tilde{\lambda}_n$
and death rates~$\tilde{\mu}_n$, coef\/f\/icients $\tilde{\pi}_n$ and transition functions~$\tilde{p}_{ij}(\cdot)$.
The processes~$\mathcal{X}$ and $\tilde{\mathcal{X}}$ are said to be
{\em similar} if there are constants $c_{ij}$, $i,j \in \mathcal{N},$ such that
\begin{gather*}
{\tilde p}_{ij}(t) = c_{ij}p_{ij}(t), \qquad i,j \in \mathcal{N}, \qquad t \geq 0.
\end{gather*}
The next theorem shows that, under certain regularity conditions, similarity imposes
strong restrictions on the birth and death rates.

\begin{Theorem}\label{similar-theorem}
Let the birth-death processes $\mathcal{X}$ and $\tilde{\mathcal{X}}$ be either
uniquely determined by their rates or minimal. If $\mathcal{X}$ and $\tilde{\mathcal{X}}$
are similar, then their birth and death rates are related as
\begin{gather}\label{rates}
\tilde{\lambda}_n + \tilde{\mu}_n = \lambda_n + \mu_n,\qquad \tilde{\lambda}_n\tilde{\mu}_{n+1} = \lambda_n \mu_{n+1}, \qquad n \in \mathcal{N},
\end{gather}
while their transition functions satisfy
\begin{gather*}
\tilde{p}_{ij}(t) = \sqrt{\frac{\pi_i \tilde{\pi}_j}{\tilde{\pi}_i \pi_j}} p_{ij}(t), \qquad i,j \in \mathcal{N}, \qquad t \geq 0.
\end{gather*}
Conversely, if $\mathcal{X}$ and $\tilde{\mathcal{X}}$ are birth-death processes
with rates related as in~\eqref{rates} then~$\mathcal{X}$ and~$\tilde{\mathcal{X}}$ are similar.
\end{Theorem}

In the more restricted setting in which $\mathcal{X}$ and $\tilde{\mathcal{X}}$
are uniquely determined by their rates the statements
of this theorem were given in \cite[Theorems~1 and~2]{L00}. Since the additional
restrictions are not used in the proof of the {\em necessity} of~\eqref{rates} for
similarity of $\mathcal{X}$ and $\tilde{\mathcal{X}}$, the question remains
whether \eqref{rates} is {\em sufficient} for similarity of~$\mathcal{X}$
and~$\tilde{\mathcal{X}}$ when the processes are not uniquely def\/ined by their rates (but minimal).
This, however, follows immediately from Karlin and McGregor's representation formula~\eqref{spectralrep},
since, considering the remarks preceding~\eqref{spectralrep}, the
polynomials and natural measure associated with $\mathcal{X}$ must be identical to
those of~$\tilde{\mathcal{X}}$ if~\eqref{rates} prevails.
Interestingly, Fralix~\cite{F15} recently established a suf\/f\/icient condition for
similarity in the setting of continuous-time {\em Markov chains} (conjectured
earlier by Pollett~\cite{P01}), which amounts to~\eqref{rates} when applied to birth-death processes.

On relating the results of Theorem~\ref{similar-theorem} to \eqref{recP} and \eqref{rep}
we see that a family of similar birth-death processes is characterized by the
fact that all members are associated with the same orthogonal polynomial
sequence~$\{P_n\}$~-- and hence with the same (natural) orthogonalizing measure~$\psi$~--
while each individual member may be characterized by the value of
$\mu_0$, which can be any number in $[0,1/m_{-1}(\psi)]$. So a family of
similar birth-death processes has either one member (if $m_{-1}(\psi) = \infty$) or
inf\/initely many members (if $m_{-1}(\psi) < \infty$). Note that there is always
a~member in the family with $\mu_0 = 0$, the {\em representative} of the family.
By \cite[equation~(2.4), Lemma~6 on p.~527]{K57a} we have
\begin{gather*}
\mu_0 = 0 \quad \Longrightarrow \quad m_{-1}(\psi) = \sum_{n=0}^\infty \frac{1}{\lambda_n\pi_n},
\end{gather*}
so to decide whether the representative of a family is the only member of the family
is, given the birth and death rates of the representative, a trivial task.

\section{Results}\label{results}

Having collected all we need, we are ready to draw conclusions.
To start with, consider a rate set $\{\lambda_n,\mu_n\}$ with $\mu_0 > 0$ and the natural
measure $\psi$ that, by Fact~\ref{fact3}, is def\/ined by this set.
Since, by Fact~\ref{fact3} again, $m_{-1}(\psi) < \infty$ and $\mu_0 \leq 1/m_{-1}(\psi)$, we
can choose $a = (\mu_0 m_{-1}(\psi))^{-1} - 1 \geq 0$ and thus link the rate set
$\{\lambda_n,\mu_n\}$ to the measure $\phi^{(a)}$ def\/ined in~\eqref{phi0} and~\eqref{phia}.

Let us f\/irst assume that the rate set $\{\lambda_n,\mu_n\}$ is such that the series \eqref{unique}
diverges, whence it corresponds to a unique birth-death process $\mathcal{X}$. Then,
by Fact~\ref{fact4}, there are two possibilities. The f\/irst is that~$\psi$ is det(S),
in which case, by Theorem~\ref{det}, $\phi^{(a)}$ is also det(S). The second
possibility is that $\psi$ is indet(S) and $m_{-1}(\psi) = 1/\mu_0$. But then $a = 0$,
so that, by Theorem \ref{det}, $\phi^{(a)}=\phi^{(0)}$ is det(S) again. So, in
any case, $\phi^{(a)}$ is det(S) and hence natural. If $a = 0$ it is possible for
$\phi^{(a)}$ to def\/ine an inf\/inite family of rate sets, but, as agreed upon in
Section~\ref{shellpolynomials}, we will always associate with $\phi^{(a)}$ the
unique rate set $\{\lambda_n^{(a)},\mu_n^{(a)}\}$ with $\mu_0^{(a)} = 0$ determined by
the parameters in the recurrence relation~\eqref{recS} for the shell polynomials~$S_n^{(a)}$.
So we have now linked the rate set $\{\lambda_n,\mu_n\}$ with $\mu_0 > 0$ to a rate set $\{\lambda_n^{(a)},\mu_n^{(a)}\}$ with
$\mu_0^{(a)} = 0$, which, by Fact~\ref{fact4}, uniquely def\/ines a birth-death process
$\mathcal{X}^{(a)}$. But on comparing~\eqref{dual} and \eqref{dualrates}, we see that
\begin{gather*}
\lambda_n^{(a)} = \lambda_n^d \qquad \mbox{and}\qquad \mu_n^{(a)} = \mu_n^d, \qquad n \geq 0,
\end{gather*}
so that, actually, $\mathcal{X}^{(a)} = \mathcal{X}^d$, the dual process of $\mathcal{X}$.
Indeed, on comparing~\eqref{phi0} and~\eqref{phia} with~\eqref{dualpsi}, we also observe that
$\phi^{(a)} = \psi^d$, as required. We summarize our f\/indings in the next theorem.
\begin{Theorem} \label{mu0>0}
Let $\{\lambda_n,\mu_n\}$ with $\mu_0 > 0$ be a rate set for which the series~\eqref{unique} diverges, $\mathcal{X}$~the birth-death process defined by this set, and~$\psi$ the corresponding measure. Then $\phi^{(a)}$, defined by~\eqref{phi0}
and~\eqref{phia}, is ${\rm det(S)}$ for all $a\geq 0$, while, for
$a = (\mu_0 m_{-1}(\psi))^{-1}-1$, it is the measure corresponding to $\mathcal{X}^d$,
the dual process of~$\mathcal{X}$.
\end{Theorem}

Note that the (not so obvious) condition $\mu_0m_{-1}(\psi) \leq 1$, which a rate
set associated with the natural measure~$\psi$ with $m_{-1}(\psi)<\infty$ should
satisfy, has a very natural counterpart for the measure of the dual process, namely $a \geq 0$.

Evidently, \eqref{dualrates} establishes a one-to-one correspondence between rate sets
with $\mu_0 = 0$ and those with $\mu_0 > 0$. So, as long as we work in the setting of
rate sets that uniquely def\/ine a~birth-death process, the above procedure mapping the
rate set $\{\lambda_n,\mu_n\}$ with $\mu_0 > 0$ to the rate set
$\{\lambda_n^{(a)},\mu_n^{(a)}\}$ with $\mu_0^{(a)} = 0$, via the corresponding
birth-death processes, must ref\/lect this correspondence.
In other words, every measure $\phi$ that corresponds to a rate set with $\mu_0 = 0$,
must be of the form~$\phi^{(a)}$ of~\eqref{phi0} and~\eqref{phia} for some $a \geq 0$,
with~$\psi$ being the measure of the dual process. Here are the details of this
correspondence.

Consider a rate set $\{\tilde{\lambda}_n,\tilde{\mu}_n\}$ with $\tilde{\mu}_0 = 0$ for which the analogue
of the series \eqref{unique} diverges. Let $\tilde{\mathcal{X}}$ be the birth-death
process uniquely def\/ined by this rate set, $\tilde{P}_n$ the corresponding polynomials
and $\tilde{\phi}$ the corresponding measure, which, in view of the analogue of~\eqref{detS}, must be det(S). Then, letting
\begin{gather*}
a = \frac{\tilde{\phi}(\{0\})}{1 - \tilde{\phi}(\{0\})} \qquad \mbox{and} \qquad
\phi^{(0)} = \frac{\tilde{\phi} - \tilde{\phi}(\{0\})\delta_0}{1 - \tilde{\phi}(\{0\})},
\end{gather*}
$\tilde{\phi}$ can be represented as
\begin{gather*}
\tilde{\phi} = \frac{1}{a+1}\big(a \delta_0 + \phi^{(0)}\big).
\end{gather*}
Def\/ining the (probability) measure $\psi$ by
\begin{gather}\label{psi}
\psi([0,x]) = \frac{1}{m_1(\phi^{(0)})}\int_{[0,x]} y\phi^{(0)}(dy), \qquad x \geq 0,
\end{gather}
we can apply some results of Berg and Thill~\cite{B91b, B91a} to conclude the following.

\begin{Lemma}
The measure $\psi$ defined by \eqref{psi} is the natural solution of the corresponding moment problem.
\end{Lemma}
\begin{proof} If $\tilde{\phi}(\{0\}) > 0$ then, by \cite[Lemma~5.4]{B91a}, $\psi$ must be det(S), and
hence natural, since $\tilde{\phi}$ is~det(S).
If $\tilde{\phi}(\{0\})=0$ (so that $\tilde{\phi} = \tilde{\phi}^{(0)}$) and~$\psi$ is
indet(S), then, by \cite[Theorem~2.4]{B91b}, the {\em density index} of
$\tilde{\phi}^{(0)}$ (the largest $n\in\mathbb{N}$ such that the polynomials are
dense in $x^n\tilde{\phi}^{(0)}(dx)$) equals~2, implying that~$\psi$ has density index~1.
Hence, by \cite[Theorem 2.1]{B91b}, $\psi$ must be natural.
\end{proof}

Evidently, $m_{-1}(\psi) < \infty$, so we see that $\psi$ has the properties
imposed on~$\psi$ in Section~\ref{introduction}. Since
\begin{gather*}
\int_{[0,\infty)} \tilde{P}_1(x) \tilde{\phi}(dx) = \int_{[0,\infty)} (x-\tilde{\lambda}_0) \tilde{\phi}(dx) = 0,
\end{gather*}
we have $m_1(\tilde{\phi}) = \tilde{\lambda}_0$, while
\begin{gather*}
\psi([0,x]) = \frac{1}{m_1(\tilde{\phi})}\int_{[0,x]} y\tilde{\phi}(dy)
	= \frac{1}{\tilde{\lambda}_0} \int_{[0,x]} y\tilde{\phi}(dy), \qquad x \geq 0,
\end{gather*}
so that $m_{-1}(\psi) = (1-\tilde{\phi}(\{0\}))/\tilde{\lambda}_0$.
We can now associate a rate set $\{\lambda_n,\mu_n\}$ with~$\psi$ by letting
\begin{gather*}
\mu_0 = \frac{1}{(a+1)m_{-1}(\psi)} = \tilde{\lambda}_0,
\end{gather*}
so that $0 < \mu_0 \leq 1/m_{-1}(\psi)$,
and choosing $\lambda_n$ and $\mu_n$ such that the polynomials $P_n$ def\/ined by~\eqref{recP}
and~\eqref{rep} are orthogonal with respect to~$\psi$. Next identifying
$\tilde{\phi}$ with the measure $\phi^{(a)}$ def\/ined in~\eqref{phi0} and~\eqref{phia},
we can identify the rates $\tilde{\lambda}_n$ and $\tilde{\mu}_n$ with the rates $\lambda_n^{(a)}$ and
$\mu_n^{(a)}$, respectively, appearing in the recurrence relation~\eqref{recS} for the
shell polynomials corresponding to the sequence $\{P_n\}$. On comparing~\eqref{dual} and \eqref{dualrates}, we thus f\/ind
\begin{gather*}
\tilde{\lambda}_n = \lambda_n^d \qquad \mbox{and}\qquad \tilde{\mu}_n = \mu_n^d, \qquad n \geq 0.
\end{gather*}
It follows that the series \eqref{dualunique}, and hence the series \eqref{unique},
diverges, so that the rate set $\{\lambda_n,\mu_n\}$
def\/ines a unique birth-death process~$\mathcal{X}$. Moreover, $\tilde{\mathcal{X}} =
\mathcal{X}^d$, the dual process of~$\mathcal{X}$. In summary, we can state the converse of Theorem~\ref{mu0>0} as follows.

\begin{Theorem}\label{mu0=0}
Let $\{\tilde{\lambda}_n,\tilde{\mu}_n\}$ with $\tilde{\mu}_0 = 0$ be a rate set for which the analogue of
the series~\eqref{unique} diverges, $\tilde{\mathcal{X}}$ the birth-death process
defined by this set, and~$\tilde{\phi}$ the corresponding measure. Then, letting
$a = \tilde{\phi}(\{0\})/(1 - \tilde{\phi}(\{0\}))$, the measure $\tilde{\phi}$ can be identified with~$\phi^{(a)}$, defined by~\eqref{phi0} and~\eqref{phia}, where $\psi$ is the natural
measure corresponding to a birth-death process $\mathcal{X}$ with
$\mu_0=((a+1)m_{-1}(\psi))^{-1}>0$. Also, $\tilde{\mathcal{X}} = \mathcal{X}^d$, the dual process of~$\mathcal{X}$.
\end{Theorem}

Still residing in the setting of birth-death processes that are uniquely def\/ined by their
rates, we recall from Section~\ref{similarbirthdeath} that a collection of birth-death processes
sharing the same natural measure $\psi$ with f\/inite moment of order~$-1$ is called a
family of similar processes. The individual members of this family are identif\/ied by the
value of~$\mu_0$, which may be any value in the interval $0 \leq \mu_0 \leq 1/m_{-1}(\psi)$.
However, in view of the preceding observations, it seems more appropriate to exclude
the process with $\mu_0=0$ from this family and view this process as a member of
a new family of birth-death processes, which all have $\mu_0=0$ and a measure of the type
\begin{gather}\label{type}
\frac{1}{a+1} (a\delta_0 + \psi ),
\end{gather}
where $a \geq 0$, the process at hand corresponding to $a=0$.

Def\/ining families of birth-death processes in this way allows us to extend the
duality concept for individual birth-death processes to families of birth-death processes.
Indeed, if $\psi_1$ is a natural measure with $m_{-1}(\psi_1)<\infty$ then there is a
one-to-one correspondence between the family of similar processes with measure
$\psi_1$ and $\mu_0>0$, and the family of processes with $\mu_0=0$ and a~measure
of the type~\eqref{type}, where $a \geq 0$ and $\psi \equiv \psi_2$ is given by
\begin{gather*}
\psi_2([0,x]) = \frac{1}{m_{-1}(\psi_1)} \int_{[0,x]}y^{-1}\psi_1(dy), \qquad x \geq 0,
\end{gather*}
in the sense that corresponding processes are each other's dual. Note that $\psi_2$ is det(S) by Theorem~\ref{det}.
If $m_{-1}(\psi_2)<\infty$, or, equivalently, $m_{-2}(\psi_1)<\infty$, we can
view $\psi_2$ as the producer of a family of similar birth-death processes with
$\mu_0>0$, which, in turn, is dual to a family of processes with $\mu_0=0$ and a measure of the type \eqref{type}, etc.

Moving beyond the setting of birth-death processes that are uniquely def\/ined by
their rates the situation becomes more complicated, but as noted in Section~\ref{dualbirthdeath} the duality concept for rate sets can be extended, provided one restricts oneself to minimal and maximal processes, and links a minimal process
to a maximal process. We next elaborate on how this af\/fects the measures involved.

So consider again a rate set $\{\lambda_n,\mu_n\}$ with $\mu_0 > 0$, and the natural
measure $\psi$ def\/ined by this set. As before we can choose
$a = (\mu_0 m_{-1}(\psi))^{-1} - 1 \geq 0$ and thus link the rate set
$\{\lambda_n,\mu_n\}$ to the measure $\phi^{(a)}$ def\/ined in~\eqref{phi0}
and~\eqref{phia}. We now assume that the rate set $\{\lambda_n,\mu_n\}$ is such that
the series \eqref{unique} {\em converges}, so that, by Fact~\ref{fact4},
$\phi^{(a)}$ is indet(S) and $a>0$. Subsequently comparing~\eqref{phia}
and~\eqref{dualpsi1} we conclude that we actually have $\phi^{(a)} = \psi_0^d$.
It follows that $\phi^{(a)}$ is $N$-extremal (as noted already in~\cite{B11}) and
corresponds to the {\em maximal} birth-death process associated with the rate
set $\{\lambda_n^d,\mu_n^d\}$ with $\mu_0^d=0$, given by~\eqref{dualrates}.
Summarizing we can state the following.

\begin{Theorem}\label{mu0>0min}
Let $\{\lambda_n,\mu_n\}$ with $\mu_0 > 0$ be a rate set for which the series~\eqref{unique}
{\em converges}, $\mathcal{X}$~the {\em minimal} birth-death process defined
by this set, and $\psi$ the corresponding measure. Then $\mu_0<1/m_{-1}(\psi)$,
and $\phi^{(a)}$, defined by~\eqref{phi0} and~\eqref{phia}, is ${\rm indet(S)}$ for $a>0$.
Moreover, for $a = (\mu_0 m_{-1}(\psi))^{-1}-1$, $\phi^{(a)}$ is the measure
corresponding to $\mathcal{X}^d$, the {\em maximal} process that is dual to~$\mathcal{X}$, and hence $N$-extremal.
\end{Theorem}

\begin{Remark}
The argument given in \cite{B11} for the fact that $\phi^{(a)}$ is $N$-extremal
is not entirely clear, but, in any case, a reference to \cite[Theorem~8]{B81}
(besides the reference to \cite[Theorem~5.5]{B91a} given in~\cite{B11}) is
suf\/f\/icient to justify the statement.
\end{Remark}

Of course there is a converse to Theorem~\ref{mu0>0min}~-- the analogue of Theorem~\ref{mu0=0}~-- which, however, we will not formulate explicitly.
It may be more interesting to look at the {\em minimal} process corresponding to
the rate set $\{\tilde{\lambda}_n,\tilde{\mu}_n\}$ with $\tilde{\mu}_0 = 0$ which does not uniquely def\/ine
a birth-death process, since the associated measure is natural. We give the result
without proof and refrain again from formulating its converse explicitly.

\begin{Theorem}\label{mu0=0min}
Let $\{\tilde{\lambda}_n,\tilde{\mu}_n\}$ with $\tilde{\mu}_0 = 0$ be a rate set for which the analogue of
the series~\eqref{unique} {\em converges}, $\tilde{\mathcal{X}}$ the {\em minimal} birth-death
process defined by this set, and $\tilde{\phi}$ the corresponding measure. Then,
letting $a = \tilde{\phi}(\{0\})/(1 - \tilde{\phi}(\{0\}))$, the measure $\tilde{\phi}$ can be
identified with $\phi^{(a)}$, defined by~\eqref{phi0} and~\eqref{phia}, where~$\psi$
is the $($natural$)$ measure corresponding to a {\em maximal} birth-death process
$\mathcal{X}$ with $\mu_0=((a+1)m_{-1}(\psi))^{-1}>0$, and hence $N$-extremal.
Also, $\tilde{\mathcal{X}} = \mathcal{X}^d$, the dual process of~$\mathcal{X}$.
\end{Theorem}

We f\/inally remark that Theorems~\ref{mu0>0} and~\ref{mu0>0min} augment the
information given in \cite[Lemma~2]{K57a}, while Theorems~\ref{mu0=0} and~\ref{mu0=0min} elaborate on \cite[Lemma~3]{K57a}.

\pdfbookmark[1]{References}{ref}
\LastPageEnding

\end{document}